\documentclass[11pt,leqno]{article}%
\usepackage{amssymb,bbm,amsmath,amsfonts,amsthm, amscd}

\usepackage{vmargin} 

\usepackage{asymptote} 

\usepackage[dvipsnames]{xcolor}

 \usepackage{mathrsfs}
\usepackage{macros}

\usepackage{enumerate}

\usepackage{color}

\usepackage{xcolor}

\usepackage[linktocpage,colorlinks=true,raiselinks]{hyperref}
\hypersetup{urlcolor=blue, citecolor=red, linkcolor=blue}

\makeatletter

\@addtoreset{equation}{section}

\begin{document}

\title{\textbf{On the Bardina's model in the whole space}}

\author{Luigi C. Berselli \thanks{Dipartimento di Matematica, Universit\`a di Pisa, Via
F. Buonarroti 1/c, I-56127, ITALY, email:~luigi.carlo.berselli@unipi.it,
URL:~http://pagine.dm.unipi.it/berselli} \and Roger Lewandowski\thanks{IRMAR, UMR 6625,
Universit\'e Rennes 1, and Fluminance team, INRIA, Campus Beaulieu, 35042 Rennes cedex
FRANCE; email:~Roger.Lewandowski@univ-rennes1.fr,
URL:~http://perso.univ-rennes1.fr/roger.lewandowski/}} 
\maketitle

\begin{abstract} 
  We consider the Bardina's model for turbulent incompressible flows in the whole space
  with a cut-off frequency of order $\alpha^{-1} >0$. We show that for any $\alpha >0$
  fixed, the model has a unique regular solution defined for all $t \in [0, \infty[$. \end{abstract}
MCS Classification: 35Q30, 35D30, 76D03, 76D05.
\smallskip


\section{Introduction} 
The purpose of this paper is the study of the Large Eddy Simulation (LES) Bardina's model
in the whole space
\begin{equation}
  \label{eq:bardina} 
  \left \{\begin{array}{ll} \zoom \p_t \uv + \div\, (\overline{\uv
        \otimes \uv}) - \nu \Delta \uv + \g p = 0 & \quad \hbox{in } \R^+\times\R^3, 
      \\ 
      \div\,\uv= 0 \phantom{\int_0^N} & \quad \hbox{in } \R^+\times\R^3, 
      \\ 
      \uv_{ t = 0} =\overline{
        \uv_0}&\quad \hbox{in } \R^3.
   \end{array} \right.
\end{equation}
In this system, $\uv = \uv (t, \x) = (u_1(t, \x), u_2(t, \x), u_3(t, \x))$ denotes the
filtered velocity of a given turbulent flow, $p = p(t, \x)$ the filtered pressure, $t \ge
0$, $\x \in \R^3$, and $\nu>0$ is the kinematic viscosity.  The filtering is obtained by
bar-operator, which is given by solving, for a given $\alpha>0$, the Helmholtz
equation~\eqref{eq:Helmholtz}
\begin{equation}
  \label{eq:Helmholtz} 
  -\alpha^2 \Delta \overline \psi + \overline \psi = \psi\quad\hbox{in } \R^3.
\end{equation}
This model is called the Navier-Stokes-Bardina-$\alpha$ model (NSEB-$\alpha$ in the
following). Introduced by Bardina, Ferziger, and Reynolds~\cite{BFR80} for weather
forecasts, it is used in many practical applications (see for instance Adams and
Stolz~\cite{AS2001} and Chow, De~{W}ekker, and Snyder~\cite{CW13}). It was studied
mathematically speaking by Ilyin, Lunasin, and Titi~\cite{ILT06} and by Layton with one of
the present authors~\cite{LL03,LL06} in the space-periodic case.  In the periodic setting,
for all $T>0$ it is proved the existence of a unique weak solution $(\uv, p)$, which
satisfies
\begin{equation*}
 \uv \in L^\infty([0, T]; H^1_{per}) \cap L^2([0, T]; H^2_{per}).
\end{equation*} 
In this paper we consider the Cauchy problem in the whole space and new difficulties arise
in the case of $\R^{3}$. To handle the problem, we revisit some extremely classical tools
of potential theory and explicit representation formulas as developed by
Leray~\cite{Ler1934} in his 1934 paper, where the path for modern research using Sobolev
spaces for the Navier-Stokes equations has been paved. We also observe that contrary to
Leray's approach, here the smoothing is not made with a mollification, but with the
solution of the differential problem~\eqref{eq:Helmholtz} and this implies weaker
estimates on the smoothed fields. Moreover, here the whole quantity $\uv \otimes \uv$ is
smoothed/filtered, while in Leray's model the filtering is applied only to the convective
field. The special properties of the Helmholtz filter~\eqref{eq:Helmholtz}, which are very
relevant in practical computations, are proved in Section~\ref{sec:helmholz_filter}.

Throughout the paper, the initial data $\uv_0\in L^2(\R^3)^3$ is given, and satisfies
$\div \, \uv_0 = 0$, so that $\overline{ \uv_0} \in H^2(\R^3)^3$ and $\div\,\overline{
  \uv_0}=0$.  Our goal is to build a unique regular solution, global in time, to the
NSEB-$\alpha$ model, for a given parameter $\alpha >0$.

\smallskip
 
As it was firstly observed in~\cite{LL03}, the key feature of the NSEB-$\alpha$ is the
following energy balance (equality),
\begin{equation}
  \label{eq:energy_estimate_Bardina} 
  \left \{ 
    \begin{array}{l} 
      \zoom {1 \over 2} \left (
        \alpha^2 \int_{\R^3} |\g \uv (t, \x) |^2 d \x + \int_{\R^3} |\uv (t, \x) |^2 d\x \right )
      + \nu \alpha^2 \int_0^t \int_{\R^3} | \Delta \uv (t', \x)|^2 \, d\x \,dt' + 
      \\
      \zoom \nu
      \int_0^t \int_{\R^3} | \g \uv (t', \x) |^2\, d\x dt' = {1 \over 2} \left ( \alpha^2
        \int_{\R^3} |\g \overline {\uv_0} (\x) |^2 d \x + \int_{\R^3} |\overline{\uv_0} (\x) |^2
        d\x \right ),
    \end{array} 
  \right.
\end{equation}
which is satisfied by any solution $(\uv, p)$, belonging to $(L^2(\R^3))^4$ with its
derivatives. This energy balance is the basic building block of our construction. Before
stating our main result, let us precisely state what we mean by ``regular solution''. The
following definition, which is largely inspired by that of regular solution
in~\cite{Ler1934}, will turn out to be well-suited to the NSEB-$\alpha$ model.
\begin{definition}
  \label{def:regular} We say that $(\uv, p)$ is a regular solution of the NSEB-$\alpha$
  \eqref{eq:bardina} over the time interval $ [0, T[$ (eventually $T = + \infty$) if
  \begin{enumerate}[i)]
  \item $\uv, \p_t \uv, \g \uv, D^2 \uv, p, \g p$ are well-defined and continuous for $(t,
    \x) \in \, [0, T[ \times \R^3$, and they satisfy the relations~(\ref{eq:bardina}.i)
    and~(\ref{eq:bardina}.ii) in $\R^3$ for all $t\in\,]0,T[$;
  \item $\forall \, \x \in \R^3$,\quad $\uv(0, \x) = \overline{\uv_0}(\x)$;
  \item $\forall \, \tau < T$,\quad  $\uv \in C([0, \tau]; H^2(\R^3)^3)$.
  \end{enumerate}
\end{definition}
It is worth noting that when $(\uv, p)$ is a regular solution, then the pressure $p$ is
solution of the Poisson equation
\begin{equation}
  \label{eq:pression} 
  \Delta p = - \div[ \div ( \overline{\uv \otimes \uv } )],
\end{equation} 
at any given time $t \in \, ]0, T[$. When $T < \infty$, we say that $(\uv, p)$ becomes
``singular when $t \to T$'' if it is a regular solution over $[0,T[$ and
\begin{equation*} 
  \lim_{\stackrel {t \to T}{t < T} } \| \uv (t, \cdot )\|_{2, 2}= \infty,
\end{equation*}
where $ \| \uv (t, \cdot) \|_{2, 2}$ denotes the $H^2(\R^{3})$ norm of $ \x \mapsto \uv(t,
\x)$ at a given time $t$.

The main result of this paper is the following.
\begin{theorem}
  \label{thm:existence_solution_reg} 
  The NSEB-$\alpha$ model~(\ref{eq:bardina}) has a unique regular solution $ (\uv, p)$
  defined for any $t \in [0, \infty[$, which satisfies the energy
  balance~(\ref{eq:energy_estimate_Bardina}), for all $t>0$, such that 
  \begin{equation*}
    \p_t \uv \in C([0, \infty[; L^2(\R^3)^3), \quad \text{and}\quad  p \in C([0, \infty[;H^4(\R^3)). 
  \end{equation*} 
  Moreover, $\forall \, \tau >0$, $\forall \, m \ge 0$,\quad  $(\uv, p) \in C([\tau, \infty[;
  H^m(\R^3)^3 \times H^m(\R^3))$.
\end{theorem} 
In particular, regular solutions of the NSEB-$\alpha$ model do not become singular in a
finite time, and we found that the model exerts a strong regularizing effect on the
pressure.

\medskip

The convergence, as $\alpha\to0^{+}$, of solutions to the NSEB-$\alpha$ model to weak
solutions of the Navier-Stokes equations will studied in a forthcoming
paper~\cite{LBRL17}.
\section{The Helmholtz filter}
\label{sec:helmholz_filter} 
This section is devoted to the study of the Helmholtz equation~(\ref{eq:Helmholtz}) and
its associated Green's kernel. We will:
\begin{enumerate}[i)]
\item set some notations;
\item introduce the Green's kernel and deduce a few basic inequalities from usual results
  about convolutions;
\item draw the link between the integral representation and the variational solutions
  of~(\ref{eq:Helmholtz}).
\end{enumerate}
\subsection{General setting}
Let $\alpha = (\alpha_1, \alpha_2, \alpha_3) \in \N^3$ be a multi-index and let $| \alpha
| = \alpha_1 + \alpha_2+ \alpha_3$, then we denote as usual
\begin{equation*} D^\alpha \uv = (D^\alpha u_1, D^\alpha u_2, D^\alpha u_3),
  \quad\text{with}\quad D^\alpha u_i = { \p ^{| \alpha | }u_i \over \p x_1 ^{\alpha_1} \p
    x_2 ^{\alpha_2 } \p x_3 ^{\alpha_3} }.
\end{equation*}
For any given $m \in \N$, when we write $D^m \uv$ we assume that $D^\alpha \uv$ is well
defined whatever $\alpha$ is a multi-index such that $| \alpha | = m$, and in practical
calculations we will use the following notation
\begin{equation*} 
  | D^m \uv | = \sup_{ | \alpha | = m } | D^\alpha \uv |.
\end{equation*}
%
space $W^{m,p}(\R^3)$ is equipped with the norm
$$
 \| w \|_{m,p} = \sum_{j=0}^m \| D^j w \|_{L^p(\R^3)},
$$
as usual $H^m(\R^3) = W^{m,2}(\R^3)$.  Throughout the paper, we will use the following
consequence of the Green's formula in $\R^3$ (see~\cite[Eq.~(1.11)]{Ler1934}):
\begin{equation}
  \label{eq:green} 
  \forall \, u \in W^{1,p}(\R^3), \forall \, v \in W^{1,p'}(\R^3), \qquad
  \int_{\R^3} \left [u (\x) {\p v \over \p x_i} (\x) + {\p u \over \p x_i} (\x) v (\x)
  \right ] \, d\x = 0,
\end{equation}
for $1 < p < \infty$, where as usual $1/p + 1/p' =1$.
\subsection{The Helmholtz kernel}
In order to prove the main estimates it turns useful to use the integral representation
formula for solutions of~\eqref{eq:Helmholtz}. Let $H_\alpha$ denotes, for $\alpha>0$ the
kernel given by
\begin{equation}
  \label{eq;filter} 
  H_\alpha (\x ) = {1 \over 4 \pi \alpha^2} {e^{-{ | \x | \over \alpha}}
    \over | \x | }\qquad\text{defined for } \x\not=\mathbf{0}.
\end{equation} 
We notice that $\| H_\alpha \|_{0, 1} = 1$ and
\begin{equation*}
  \forall \, \x \not=0, \qquad - \alpha^2 \Delta H_\alpha (\x) + H_\alpha(\x) = 0,
\end{equation*} 
which leads to the following result
\begin{lemma} 
  Let $\varphi \in C^\infty_c(\R^3)$. Then
  \begin{equation}
    \label{eq:strong} 
    \forall \, \x \in \R^3, \qquad \int_{\R^3} (-\alpha^2 \Delta \varphi ({\bf y})+ \varphi
    ({\bf y}) ) H(\x-{\bf y})\, d{\bf y} = \varphi ( \x). 
\end{equation}
\end{lemma} 
The classical proof is very close to that of the Green's representation formula in Gilbarg
and Trudinger~\cite[Ch.~2]{GT01}, so we skip the details. Thereby, we get
\begin{equation*}
  \forall\,\alpha>0\qquad-\alpha^2 \Delta H_\alpha + H_\alpha = \delta_0,
\end{equation*} 
in the sense of distributions over $\R^3$, where $\delta_0$ denotes the Dirac (delta)
measure centered at the origin.

\smallskip
 
Let $\psi$ be any measurable function and let $\overline \psi$ denote the filtered
function:
\begin{equation}
  \label{eq:helmconv} 
  \overline \psi (\x) = H_\alpha \star \psi (\x) = \int_{\R^3} H_\alpha
  (\x - {\bf y}) \psi ( {\bf y})\, d{\bf y},
\end{equation} 
as long as the integral converges.  We deduce from the Cauchy-Schwarz and Young
inequalities the series of formal inequalities:
\begin{eqnarray} 
  && 
  \label{eq:estimates_1} 
  \|\overline \psi \|_{0, p} \le \| \psi \|_{0,
    p}, \quad \text{for all }1 \le p \le \infty, 
  \\
  && \label{eq:estimates_2} 
  \| \g \overline
  \psi \|_{0, 2} \le {2 \over \alpha} \| \psi \|_{0, 2}, 
  \\ 
  && \label{eq:norm_infty_bound}
  \| \overline \psi \|_{0, \infty} \le {1 \over \sqrt {8 \pi} \alpha^{3 \over 2} } \| \psi
  \|_{0, 2}, 
  \\ && \label{eq:norm2_1} 
  \|\overline \psi \|_{0, 2} \le {1 \over \sqrt {8 \pi}
    \alpha^{3 \over 2} } \| \psi \|_{0, 1}.
\end{eqnarray}
\begin{Remark}
  \label{rem:reg_helm} 
  We notice that $H_\alpha \in L^q(\R^3)$, for all $q<3$, and it is continuous for
  $\x\not=0$. Then, if $\psi \in L^2(\R^3)$, we have that $\overline \psi \in C^0(\R^3)$,
  but it does not necessary belong to $C^1(\R^3)$. On the other hand, when $\psi \in
  L^p(\R^3)$ for some $p>3$, then as $\g H_\alpha \in L^r(\R^3)$, $r<3/2$ and is
  continuous for $\x \not=0$, we obtain that $\overline \psi \in C^1(\R^3)$.
  
  In any case, assuming just that $\psi \in L^q(\R^3)$ it is not possible to prove that
  $\overline \psi \in C^2(\R^3)$ since $D^2 H_\alpha \notin L^p(\R^3) $, whatever the
  value of $p>1$ is taken. However, when $\psi \in H^1(\R^3)$, then we have $\overline
  \psi \in C^1_b (\R^3) \cap H^2(\R^3)$.
\end{Remark}
\subsection{Variational formulation and related properties}
With the estimates~(\ref{eq:estimates_1}), the relation~(\ref{eq:strong}) combined with
the rotational symmetry of the kernel $H_\alpha$ and the Green's formula~(\ref{eq:green}),
we obtain:
\begin{lemma}
  \label{lem:weak_sol_Helm} 
  Let $\psi \in L^2(\R^3)$. Then $ \overline \psi \in H^{1}(\R^{3)}$ is the unique weak
  solution to the Helmholtz equation~(\ref{eq:Helmholtz}), in the following sense:
  \begin{equation}
    \label{eq:weak_Helmhoz} 
    \forall \, \varphi \in H^1(\R^3), \quad \alpha^2 \int_{\R^3}
    \g \overline \psi (\x) \cdot \g \varphi (\x) \, d\x + \int_{\R^3} \overline \psi (\x) \,
    \varphi (\x) \, d\x = \int_{\R^3} \psi (\x) \, \varphi (\x) \, d\x .
  \end{equation}
\end{lemma}
\begin{proof} 
  The existence and uniqueness of a solution to the variational
  problem~(\ref{eq:weak_Helmhoz}) is a consequence of the Lax-Milgram Theorem in
  $H^1(\R^3)$. We must check that this coincides with the convolution
  formula~(\ref{eq:helmconv}). Let $\varphi \in C^\infty_c(\R^3)$; by Fubini's Theorem and
  the Green's formula, we have
  \begin{equation*}
    \begin{array}{l}
      \zoom \alpha^2 \int_{\R^3} \g \overline \psi (\x) \cdot \g \varphi (\x)
      \, d\x + \int_{\R^3} \overline \psi (\x) \, \varphi (\x) \, d\x 
      \\
      = \zoom \int_{\R^3}
      \int_{\R^3} [\alpha^2 \g_\x H_\alpha(\x-{\bf y}) \cdot \g \varphi (\x) + H_\alpha (\x-{\bf
        y}) \varphi (\x) ] \psi ({\bf y})\, d{ \bf y} d\x 
      \\
      =\zoom \int_{\R^3} \psi ({\bf y}) d{
        \bf y} \left ( \int_{\R^3} [\alpha^2 \g_\x H_\alpha(\x-{\bf y}) \cdot \g \varphi (\x) +
        H_\alpha (\x-{\bf y}) \varphi (\x)]\, d\x \right )
      \\ 
      = \zoom \int_{\R^3} \psi ({\bf y}) d{
        \bf y} \left ( \int_{\R^3} [-\alpha^2 \Delta \varphi (\x) + \varphi (\x) ] H_\alpha ({\bf
          y}-\x)\,d\x \right ) = \int_{\R^3} \psi ({\bf y}) \varphi ({\bf y}) \,d{ \bf y},
 \end{array}
\end{equation*}
where we have used~(\ref{eq:strong}) and the fact that $H_\alpha (\x) = H_\alpha (-\x) $.
We conclude by a density argument\footnote{Applying Lemma 1.1 in Galdi and
  Simader~\cite{GS90}, we deduce that any $\varphi \in H^1(\R^3)$ goes to zero at
  infinity. So density of $C_c^\infty(\R^3)$ can be obtained by mollifying and
  truncating.}.
\end{proof}
The following corollary is straightforward.
\begin{corollary} 
  The bar operator is self-adjoint with respect to the $L^2$-scalar product
  $(\,.\,,\,.\,)$, which means:
\begin{equation*}
  \forall \, \psi, \varphi \in L^2(\R^3), \qquad (\overline \psi, \varphi)
  = ( \psi, \overline \varphi).
\end{equation*}
\end{corollary} 
The convergence of $\overline \psi$ to $\psi$ as $\alpha \to 0^{+}$ is the aim of the next
lemma.
\begin{lemma}
  \label{lem:estimate_convergence} 
  Let $1 \leq p \le \infty$ and let $\psi \in W^{2,p}(\R^3)$. Then
  \begin{equation}
    \label{eq:est_hausdorff}
    \| \overline \psi - \psi \|_{0, p} \le \alpha^2 
    \|\Delta \psi \|_{0, p}.
  \end{equation}
\end{lemma}
 \begin{proof} 
   It is enough to prove the estimate~(\ref{eq:est_hausdorff}) when $\psi \in
   C_c^\infty(\R^3)$. In this case, we deduce from~(\ref{eq;filter})
   and~(\ref{eq:helmconv}) that $\overline \psi, \g\overline \psi = \mathcal{O}(e^{-{ | \x
       |\over \alpha }} )$, which allows the following integrations by parts. Let $\delta
   \psi = \overline \psi - \psi$, that satisfies the equation
   \begin{equation}
     \label{eq:deltapsi} 
     - \alpha^2 \Delta \delta \psi + \delta \psi = \Delta \psi.
   \end{equation} 
   Let $1 \le p < \infty$ be given. We take $\delta \psi | \delta \psi |^{p-2}$ as test
   function in the equation~(\ref{eq:deltapsi}), which yields\footnote{Strictly speaking
     we should first take $\delta \psi ( \E + \delta \psi^2 )^{(p-2)/2} $ as test
     function, and then pass to the limit when $\E \to 0^+$ once the estimate is
     established. This is standard, so that we skip the details. The reader can see this
     for instance in Di Perna-Lions~\cite{DL89}, where similar calculations are carried
     out.}
   \begin{equation*}
     (p-1)\,\alpha^2 \inte |\g \delta \psi |^2 | \delta \psi |^{p-2} + \inte |\delta \psi |^p \le
     \alpha^2 \inte |\Delta \psi |\, | \delta \psi |^{p-1}.
 \end{equation*}
 The estimate~(\ref{eq:est_hausdorff}) follows 
From the H\"older inequality we get
   \begin{equation*}
\|\delta \psi \|_{0,p}^p \le
     \alpha^2 \|\Delta \psi \|_{0,p} \| \delta \psi \|^{p-1}_{0,p}\qquad 1\leq p<+\infty,
 \end{equation*}
 from which we get the thesis when $1 \le p < \infty$.  As $\psi \in C^\infty_c$, then
 $\|\psi\|_{0,\infty}=\lim_{p\to+\infty}\|\psi\|_{0,p}$ and in view of the decay of
 $\overline \psi$ at infinity, we get, by passing to the
 limit when $p \to \infty$ in the right-hand side of~(\ref{eq:est_hausdorff}),
 \begin{equation*}
   \limsup_{p\to+\infty}  \| \overline \psi - \psi \|_{0, p} \le \alpha^2 
   \|\Delta \psi \|_{0, \infty},
 \end{equation*}
which proves the estimate also in the limit case.
\end{proof}
From Lemma~\ref{lem:estimate_convergence} and a straightforward density argument, we get:
\begin{lemma} 
  Let $1 \le p < \infty$, $\psi \in L^p(\R^3)$. Then $\overline \psi \to \psi$ in
  $L^p(\R^3)$ when $\alpha \to 0^{+}$.
\end{lemma} 
We also will need the following Leibniz like formula:
\begin{lemma}
  \label{lem:differentiation} 
  Let $\varphi, \psi \in H^1(\R^3) \cap L^\infty(\R^3)$. Then
  \begin{equation}
    \label{eq:differentiation} 
    D \,\overline{\varphi \psi } = \overline {\psi D \varphi}+
    \overline {\varphi D \psi},
\end{equation} 
where $D$ denotes any partial derivative $\frac{\partial}{\partial x_i}$, $i=1,2,3$.
\end{lemma}
\begin{proof} 
  According to the assumptions about $\varphi$ and $\psi$, the products $\varphi\, \psi$,
  $\psi\, D \varphi$ and $\varphi\, D \psi$ are all in $L^2(\R^3)$. Therefore, we deduce
  from Lemma~\ref{lem:weak_sol_Helm} that (at least in a weak sense),
\begin{equation*} 
  -\alpha^2 \Delta\overline{\varphi \psi } + \overline{\varphi \psi } =
  \varphi \psi,
\end{equation*}
which yields by differentiating (at least in the sense of the distributions),
\begin{equation*} 
  -\alpha^2 \Delta D\overline{\varphi \psi } + D\overline{\varphi \psi } =
  \varphi D \psi + \psi D \varphi.
\end{equation*} 
Moreover, again by Lemma~\ref{lem:weak_sol_Helm}, we have
\begin{equation*}
  -\alpha^2 \Delta (\overline {\psi D \varphi})+ \overline {\psi D
    \varphi} = {\psi D \varphi} \quad\text{and}\quad -\alpha^2 \Delta (\overline {\varphi D
    \psi})+ \overline {\varphi D \psi} = {\varphi D \psi},
\end{equation*} 
hence~(\ref{eq:differentiation}) follows, due to the uniqueness of the solution.
\end{proof} 
Finally, by applying the basic elliptic regularity to the Helmholtz
equation~(\ref{eq:Helmholtz}), we also have the estimate
\begin{equation}
  \label{eq:est_der_sec} 
  \| \overline \psi \|_{2, 2}\le {C \over \alpha}\| \psi \|_{0,2},
\end{equation}
and more generally, according to the standard elliptic theory (see, e.g,
Br\'ezis~\cite{HB83}), we always have 
\begin{equation}
  \label{eq:est_der_sob_gen} 
  \forall\,s \ge 0,\qquad \| \overline \psi \|_{s+2, 2}\le {C \over \alpha}\|
  \psi\|_{s,2}. 
\end{equation}
As a consequence we have the following result.
\begin{lemma} 
  Let $\suite \psi n$ be a sequence that converges to $\psi$ in $H^s(\R^3)$. Then
  $(\overline \psi_n)_{n \in \N}$ converges to $\overline \psi$ in $H^{s+2} (\R^3)$.
\end{lemma}
\section {A priori estimates and energy balance} 
\label{sec:GF}
To avoid repetition, we will assume throughout the rest of the paper that $\uv_0 \in
L^2(\R^3)^3$ and $\div \, \uv_0 = 0$. Regular solutions to the NSEB-$\alpha$ model are
defined in Definition~\ref{def:regular}.  For any fixed time $t$, then $\| \uv (t, \cdot)
\|_{s,p}$ denotes the norm in $W^{s,p}(\R^3)$ of the field $\x \mapsto \uv(t, \x)$ for a
given fixed $t$. \smallskip

Throughout this section, $(\uv, p)$ denotes an {\sl \`a priori} regular solution to the
NSEB-$\alpha$ model. We aim to figure out the optimal regularity of this solution and to
show that it satisfies the energy balance~(\ref{eq:energy_estimate_Bardina}). To do so, we
will:
\begin{enumerate}[i)]
\item precise few notations and practical functions linked to norms of $\uv$;
\item give the Oseen's integral representation for the calculation of the velocity $\uv$;
\item deduce from this representation additional regularity for $\p_t \uv$ and $p$. in
  order to get the energy balance;
\item find $H^m$ estimates for $(\uv, p)$ on the interval $[\tau, T[$ for any $0 < \tau <
  T$.
\end{enumerate} 
Finally, as the pressure $p$ is linked to the velocity $\uv$ by the
equation~(\ref{eq:pression}), we sometime will refer to $\uv$ as the solution of the
NSEB-$\alpha$ model, without talking about $p$.
\subsection{A few notations} Let $W(t)$ and $J(t)$ denote the following functions:
\begin{equation*}
W(t):=\| \uv (t, \cdot) \|_{0, 2}^2
  \qquad 
 \text{and}\qquad J(t) =: \|\g \uv (t, \cdot) \|_{0, 2}.
\end{equation*} 
At this stage, they could be not finite for some positive $t$.  The energy
balance~(\ref{eq:energy_estimate_Bardina}) is related to the function $E_{\alpha} (t)$
defined as follows
\begin{equation*} 
  E_\alpha (t) := \alpha^2 \int_{\R^3} |\g \uv (t, \x) |^2 d \x +
  \int_{\R^3} |\uv (t, \x) |^2 d\x = \alpha^2 J^2(t) + W(t),
\end{equation*} 
and we set
 \begin{equation*} 
   E_{\alpha, 0} :=\alpha^2 \int_{\R^3} |\g \overline {\uv_0} (\x) |^2 d
   \x + \int_{\R^3} |\ \overline {\uv_0} (\x) |^2 d\x.
 \end{equation*} 
 The following inequalities hold true:
\begin{equation*} 
  W(t) \le E_\alpha (t) \quad\text{and}\quad J(t) \le \alpha^{-1}
  \sqrt{E_\alpha (t) },
\end{equation*} 
and the energy balance~(\ref{eq:energy_estimate_Bardina}) can be rewritten as follows
\begin{equation}
  \label{eq:energy_estimate_Bardina_2} 
  {1 \over 2} E_\alpha(t) + \int_0^t \left ( \alpha^2
    \| \Delta \uv (t', \cdot) \|_{0, 2} ^2+ \nu J(t')^2 \right ) \,dt' = {1 \over 2} E_{\alpha,
    0}.
\end{equation} 
In particular, each solution $\uv$ of the NSEB-$\alpha$ model for
which~(\ref{eq:energy_estimate_Bardina_2}) holds is such that the function $t \mapsto E_\alpha
(t)$ is non-increasing and it satisfies
\begin{equation*}
  \forall t \in [0,T], \quad E_\alpha (t) \le E_{\alpha, 0} \le 5 \|
  \uv_0 \|_{0, 2}^2,
\end{equation*} 
where the latter inequality is deduced from the estimates~(\ref{eq:estimates_1})
and~(\ref{eq:estimates_2}). In the sequel we will use $E_{\alpha, 0}$ as control parameter
for the NSEB-$\alpha$ rather than $\| \uv_0 \|_{0, 2}^2$.
\subsection{Oseen representation}
Let be given any Navier-Stokes-like system in $\R^3$ of the form
\begin{equation*}
  \left \{ 
    \begin{array}{ll} 
      \p_t \uv + B(\uv, \uv) - \nu \Delta
      \uv + \g p = 0 & \quad \hbox{in } \R^+\times\R^3, 
      \\ 
      \div \, \uv = 0, & \quad \hbox{in }
      \R^+\times\R^3, 
      \\
      \uv_{ t = 0} = F({ \uv_0}) & \quad \hbox{in } \R^3.
    \end{array} \right.
\end{equation*}
In the case of the NSEB-$\alpha$ model,
\begin{equation*} 
  F(\uv_0) = \overline{\uv_0} \quad\text{and}\quad B(\uv, \uv) =\overline{
    \div(\uv\otimes \uv) }= \overline{ (\uv\cdot \g ) \uv }.
\end{equation*}
Modern analysis often describes a regular solution to this system by the abstract
differential equation
\begin{equation}
  \label{eq:abstract_formaulation} 
  \uv(t)= e^{-\nu t \Delta} F({\uv_0}) + \int_0^t e^{-\nu
    (t-t') \Delta} P B(\uv(t'), \uv(t')) \,dt',
\end{equation}
$P$ being the Leray's projector on $L^2$ divergence-free vector fields, where, given any
vector field $V=(V_1, V_2, V_3)$, then
$$
 P V = (P V_1, P V_2, P V_3),\quad\text{where}\quad P V_i=V_i-\Delta^{-1} \p_i \p_j V_j.
$$
This abstract formulation has been extensively exploited in the case of the Navier-Stokes
equations, for which $B(\uv, \uv) = (\uv \cdot \g) \, \uv$, $F(\uv_0) = \uv_0$, probably
starting from the famous paper by Fujita and Kato~\cite{FK64}. However, according to
Tao~\cite{TT16}, it seems that this approach has reached its limits. Unfortunately, we
have not found and alternate formulation that will revolutionize the field, yet.

On the contrary, we will be very old-fashioned in using the Oseen representation formula
which gives an explicit expression by convolutions in space of the integral
relation~(\ref{eq:abstract_formaulation}) through a ``semi singular" kernel. This kernel
was first determined by Oseen~\cite{CO11} for the evolutionary Stokes problem, and
developed by Leray~\cite{Ler1934} to study the Navier-Stokes equations. We introduce in
what follows the Oseen's kernel.

Let us consider the evolutionary Stokes problem with a continuous source term ${\bf f}$
and a continuous initial data $\vv_0$:
\begin{equation} 
  \label{eq:Stokes_Problem} 
  \left \{
    \begin{array}{ll}
      \p_t \vv - \nu
      \Delta \vv + \g q = {\bf f}& \quad \hbox{in } \R^+\times\R^3, 
      \\
      \div \, \vv = 0& \quad
      \hbox{in } \R^+\times\R^3, 
      \\
      \vv_{t=0} = \vv_0& \quad \hbox{in } \R^3.
    \end{array} \right.
\end{equation} 
It is well-known (see Oseen~\cite{CO11, CWO27}) that there exists a tensor ${\bf T}=(T_{i
  j})_{1 \le i, j \le 3}$ such that, being $(\vv, q)$ a regular solution
of~(\ref{eq:Stokes_Problem}), then
\begin{equation*} 
  \vv (t, \x) = (Q \star \vv_0) (t, \x) + \int_0^t \inte {\bf T} (t-t', \x
  - \yv) \cdot {\bf f}(t', \yv) \, d \yv,
\end{equation*}
where
\begin{equation*}
  Q(t, \x) := {1 \over (4 \pi \nu t)^{3/2} }e^{- {| \x |^2 \over 4 \nu t}},
\end{equation*} 
is the heat kernel and
\begin{equation*} 
  (Q \star \vv_0) (t, \x) := \inte Q(t, \x - \yv ) \vv_0(\yv) d\yv.
\end{equation*}
The components of ${\bf T}$ are detailed for
instance in~\cite{RL17}, where the following estimates are proved:
\begin{equation*}
  \forall\, m \ge 0,\,
  \exists C_{m}>0: \quad | D^m {\bf T}(t, \x ) | \le {C_m
    \over (| \x |^2 + \nu t)^{m+3 \over 2}}\qquad\forall\,(t,\x)\not=(0,{\bf 0}),
\end{equation*}
$C_m$ being some constant depending only on $m\in\N$. As a consequence we have (see for
instance~\cite{RL17}) the following result.
\begin{lemma} 
  Let $ t > 0$. Then, $\forall \, t' \in [0, t[$, the tensor field $\x' \mapsto {\bf T} (t-t',
  \x' )$ belongs to $L^1_{\x'}(\R^3)$. Moreover, $t' \mapsto \| \g \, {\bf T} (t-t', \cdot )
  \|_{0, 1} \in L^1([0, t])$ and
\begin{equation}
  \label{eq:calcul_integral}
  \| \g \, {\bf T} (t-t', \cdot ) \|_{0, 1} \le { C \over
    \sqrt{\nu (t-t')} } .
\end{equation}
\end{lemma}

\medskip

Lemma~8 in~\cite{Ler1934} applies also to the case of the NSEB-$\alpha$ model (we skip the
details), and we have also the following result
\begin{lemma}
  Let $(\uv,p)$ be a regular solution of the NSEB-$\alpha$ model~(\ref{eq:bardina}) over
  the time interval $[0, T[$, then for all $t \in [0, T[$,
  \begin{eqnarray}
    \label{eq:oseen_rep} 
    && \uv (t, \x) = (Q \, \star \, \overline{ \uv_0}) (t, \x) + \int_0^t
    \inte \g {\bf T} (t-t', \x - \yv) : \overline{\uv \otimes \uv} (t', \yv)\, d \yv dt',
    \\
    \label{eq:pression_reg}
    && p (t, \x) = {1 \over 4 \pi} \int_{\R^3} \g^2 \left ( {1 \over
        r} \right )\overline{\uv \otimes \uv}(t, \yv) \, d\yv,
  \end{eqnarray} 
  where $r := | \x - \yv |$.
\end{lemma}
In formula~(\ref{eq:pression_reg}), $\g^2 \left ( {1 \over r} \right ) $ is a Dirac tensor
so that the integral ~(\ref{eq:pression_reg}) must be understood as a singular operator,
which satisfies the assumptions of the Calder\'on-Zygmund Theorem (see Stein~\cite{ES70}
for a general setting and Galdi~\cite{GG00} for the implementation within the
Navier-Stokes equations framework).
\subsection{Regularity and energy balance}
Recall that the notion of regular solution is given in Definition~\ref{def:regular}.  The
goal is to prove that any regular solution satisfies the energy balance.  We know
that a regular solution has $H^2$-regularity, which is {\sl \`a priori} not enough to get
this energy balance. We also need extra integrability conditions for $\p_t \uv$ and $\g
p$, which is the main goal of this section. We will observe that the model exerts a strong
regularizing effect of the pressure, even near $t=0$.
\begin{lemma}
  \label{lem:pression_regularity}
  For all $0 < \tau <T$,
  \begin{equation} 
    \label{eq:reg_pressure}
    p \in C([0, \tau]; H^4(\R^3)) 
  \end{equation}
\end{lemma}
\begin{proof} 
  Here $\tau \in \, ]0, T[$ and $t \in [0, \tau]$, and we split the proof into 3 steps:
\begin{enumerate} [i)]
\item \label{it:step1} $H^4$ regularity of $\overline{\uv \otimes \uv }$ from the
  Helmholtz equation;
\item \label{it:step2} $L^2$ regularity of $p$ by using the integral
  representation~(\ref{eq:pression_reg}) and the Calder\'on-Zygmund Theorem;
\item \label{it:step3} $H^4$ regularity of $p$ by using the equation~(\ref{eq:pression})
  and the elliptic theory.
\end{enumerate} Step~\ref{it:step1}): Since $H^2(\R^{3})$ is an algebra and $\uv \in C([0,\tau];
H^2(\R^3)^3)$, then $\uv \otimes \uv \in C([0,\tau]; H^2(\R^3)^9)$. Therefore,
by~(\ref{eq:est_der_sob_gen}), it follows that $\overline{\uv \otimes \uv} \in C([0,\tau];
H^4(\R^3)^9)$ and we have
\begin{equation}
  \label{eq:H4_regularity}
  \| \overline{\uv \otimes \uv}(t, \cdot) \|_{4,2} \le {C \over
    \alpha^2}\|\uv (t, \cdot) \|_{2,2}^2.
\end{equation}
We also deduce $ \div[ \div ( \overline{\uv \otimes \uv } )] \in C([0,\tau];
H^2(\R^3)^9)$, which will be useful in step~\ref{it:step3}).

\medskip

Step~\ref{it:step2}): The $L^2$-regularity of $p$ is a consequence of the integral
representation~(\ref{eq:pression_reg}), Calder\'on-Zygmund Theorem,
and~(\ref{eq:H4_regularity}). The time continuity with values in $L^2(\R^3)$ is
straightforward, and we have in particular
\begin{equation*}
  \| p (t, \cdot) \|_{0, 2}\le {C \over \alpha^2} \|\uv (t, \cdot)\|_{2,2}^2.
\end{equation*}
Step~\ref{it:step3}): We use the equation~(\ref{eq:pression}) for the pressure, that we
write under the form
\begin{equation}
  \label{eq:pression_2}
  -\Delta p + p = F ,
\end{equation} where
\begin{equation} 
\label{eq:F_pour_la_pression}
F = \div[ \div ( \overline{\uv \otimes \uv
} )] + p \in C([0, \tau]; L^2(\R^3)), 
\end{equation}
by the results of the two previous steps. Then~(\ref{eq:F_pour_la_pression}) is a
consequence of the standard elliptic theory (see Br\'ezis~\cite{HB83}) and yields $p \in
C([0,\tau]; H^2(\R^3))$. Finally as $\div[ \div ( \overline{\uv \otimes \uv } )] \in
C([0,\tau]; H^2(\R^3)^9)$ by step~\ref{it:step1}), then $F$ given
by~(\ref{eq:F_pour_la_pression}) is in $C([0,\tau]; H^2(\R^3)^9)$, hence $p \in
C([0,\tau]; H^4(\R^3)^9)$ from~(\ref{eq:pression_2}) and we have in particular
\begin{equation*} 
  \|p(t, \cdot)\|_{4, 2}\le {C \over \alpha^2 } \| \uv (t, \cdot) \|_{2,2}^2,
\end{equation*} 
which is optimal.
\end{proof} 
We have also the following result.
 \BL
 \label{lem:velocity_regularity} For all $0 < \tau_1< \tau_2 < \infty$, $\uv \in C
 ([\tau_1, \tau_2]; H^4(\R^3)^3)$.  
\EL
\begin{proof} 
  Let $m= 3, 4$, $0 < \tau_1 < \tau_2 < T$. An argument similar to that of Lemma~3.2
  in~\cite{RL17} shows that we can differentiate under the integral sign in the Oseen's
  representation and therefore we skip the details. We get
  $$ 
  D^m \uv (t, \x) = D^{m-2} (Q \, \star \, D^2\overline{ \uv_0}) (t, \x) + \int_0^t \inte \g
  {\bf T} (t-t', \x - \yv) : D^m (\overline{\uv \otimes \uv}) (t', \yv)\, d \yv dt',
  $$
  hence by standard results about the heat kernel, Cauchy-Schwarz and Young inequalities,
  inequalities~(\ref{eq:calcul_integral}) and~(\ref{eq:H4_regularity}), and the fact that
  $H^2(\R^{3})$ is an algebra, we obtain
 \begin{equation*} 
   \| D^m \uv (t, \x) \|_{0,2} \le C \left ( {1 \over (\nu t)^{m-2 \over
         2} } \| (Q \star D^2 \overline {\uv_0})(t, \cdot)\|_{0,2} + {1 \over \alpha^2 } \int_0^t {
       \| \uv (t', \cdot) \|_{2,2} ^2 \over \sqrt { \nu (t-t') }\,dt' } \right ),
\end{equation*} 
which gives
\begin{equation*}
   \| D^m \uv (t, \x) \|_{0,2} \le C \left ( {1 \over \alpha (\nu
       \tau_1)^{m-2 \over 2} } \| \uv_0 \|_{0,2} + {1 \over \alpha^2 } \sup_{t' \in [0, \tau_2]
     }\| \uv(t', \cdot) \|_{2,2}^2 \sqrt { \nu t \over \nu } \right ),
 \end{equation*} 
 hence the result, the continuity being a consequence of standard results in analysis.
\end{proof}
We have some relevant corollaries of the previous results
\begin{corollary} 
  For all $0 <\tau<T$, $ \p_t\uv\in C([0, \tau]; L^2(\R^3)^3)$, and for all
  $0<\tau_1<\tau_2<T$, it follows that $\p_t \uv \in C ([\tau_1, \tau_2]; H^2(\R^3)^3)$.
\end{corollary}
\begin{proof} 
  It is enough to write
\begin{equation*}
\p_t \uv = - \g p + \nu \Delta \uv - \div (\overline{\uv \otimes \uv}),
\end{equation*}
and to apply the previous results to the right-hand side.
\end{proof}
\begin{corollary}
  \label{cor:energy_balance} 
  Let $\uv$ be a regular solution to the NSEB-$\alpha$ model. Then, the velocity $\uv$
  satisfies the energy balance~(\ref{eq:energy_estimate_Bardina_2}).
\end{corollary}
\begin{proof} 
  Since for all $t \in [0, T[$ we have that $\uv (t, \cdot) \in H^2(\R^3)^3$, we can take
  $ -\alpha^2 \Delta \uv + \uv$ as test vector field in~(\ref{eq:bardina}) and integrate
  over $\R^3$ by using the Stokes formula. In particular, we have (as $H^2(\R^{3})$ is an
  algebra) that $(\uv \otimes\uv)(t, \cdot) \in H^2(\R^3)^9$, hence $(\overline {\uv
    \otimes \uv })(t, \cdot) \in H^4(\R^3)^9$. Therefore, since the bar operator is
  self-adjoint, the following equalities hold true
\begin{equation*}
  ( \div (\overline {\uv \otimes \uv }), -\alpha^2 \Delta \uv + \uv) = (
  \div ( {\uv \otimes \uv }), -\alpha^2 \Delta \um + \um) = ( \div ( {\uv \otimes \uv }),
  \uv) = 0.
\end{equation*} 
Moreover, let $t \in \, ]0, T[$. As $\uv, \p_t \uv \in C([0, t+\E]; L^2(\R^3)^3)$, where
$\E >0$ is such that $t+\E <T$, we obtain by applying a standard results (see
Temam~\cite{RT01} for instance),
\begin{equation*}
(\p_t \uv(t, \cdot) , \uv(t, \cdot) = {d \over 2 dt} \inte | \uv(t, \x)|^2 d \x.
\end{equation*}
By the previous results, $\p_t \uv \in C([\E, t+\E]; H^1(\R^3)^3)$, $\p_t \g \uv \in
C([\E, t+\E]; L^2(\R^3)^9)$, $\g \uv \in C([\E, t+\E]; L^2(\R^3)^9)$. Therefore, we also
have
\begin{equation*} 
  \forall\,t\in]0,T[\qquad   -(\p_t \uv (t, \cdot), \Delta \uv (t, \cdot) ) = {d \over 2 dt} \inte
  |\g \uv (t, \x) |^2 d \x.
\end{equation*} 
Finally, since $\div \,\uv = \div\, (\Delta \uv) = 0$, and due to the integrability
results of Lemma~\ref{lem:pression_regularity} and Lemma~\ref{lem:velocity_regularity}
about $(\uv,p)$ we also have $(\g p, -\alpha^2 \Delta \uv + \uv) = 0$. The rest of the
proof is now straightforward.
\end{proof}
\begin{Remark} 
  Once the regularity of any regular solution is established and the energy balance is
  checked, uniqueness follows from the same process, based on energy balances by the
  regularity results, combined with an application  of Gronwall's lemma. One has only to
  reproduce what is written in Section~2.2 of~\cite{LL06}, with minor modifications. To
  avoid repetitions, we skip the details.
\end{Remark}
\section{Construction of a regular solution}
In this section, we start with a continuation principle, that states that a regular
solution on a finite time interval $[0, T[$ (with $T<\infty$) cannot develop any
singularity when $t \to T^{-}$ and can be extended by continuity up to time $T$. Then, we
set up the standard iteration process based on the integral Oseen representation
formula. Finally, we will construct a regular solution, obtained as the limit of the
iterations previously analyzed.
\subsection{Continuation principle}
The regular solution that will be constructed by the Picard (iteration) theorem is
local-in-time.  The results in this subsection are essential to understand the transition
from a local-in-time regular solution, to a regular solution defined for all time
$t\in[0,\infty[$, that we call a global time solution.
\begin{lemma}
  \label{lem:regularity_vitesse}
  Assume $T<\infty$. Then a regular solution to~\eqref{eq:bardina} on $[0, T[$ do not
  develop a singularity when $t \to T^{-}$.
\end{lemma}
\begin{proof}
  We must prove that $\|\uv (t, \cdot) \|_{2,2}$ remains bounded on the time interval
  $[0,T[$.  We deduce from the integral representation formula and from the
  inequality~(\ref{eq:calcul_integral}) that
 \begin{equation}
   \label{eq:borne_H_2}
   \| D^m \uv(t, \cdot) \|_{0, 2} \le {C
     \over \alpha} \| \uv_0 \|_{0,2} + \int_0^t { \| D^m (\overline{\uv \otimes \uv})
     (t', \cdot) \|_{0,2} \over \sqrt{\nu (t-t') }}\,dt'\quad \text{for }
   m=0,1,2,  
 \end{equation} 
 and, as soon as $m \le 2$, we get by~(\ref{eq:est_der_sec}), with the Cauchy-Schwarz and
 Sobolev inequalities, that
  \begin{equation*} 
    \begin{array}{ll}
      \| D^m (\overline{\uv \otimes \uv}) (t', \cdot)
      \|_{0,2} & \zoom \le C \alpha^{-1} \| (\uv \otimes \uv) (t', \cdot) \|_{0, 2} 
      \\
      & \zoom
      \le C \alpha^{-1}\| \uv (t', \cdot) \|_{0,4}^2 
      \\ 
      & \zoom \le C \alpha^{-1}\| \uv (t',
      \cdot) \|_{1,2}^2 
      \\
      &\le C \alpha^{-3} E_\alpha (t),
 \end{array} 
\end{equation*}
which leads to, by the energy balance,
 \begin{equation*} 
\| D^m (\overline{\uv \otimes \uv}) (t', \cdot) \|_{0,2} \le {C
\alpha^{-3}}E_{\alpha, 0},
\end{equation*} 
that we combine with~(\ref{eq:borne_H_2}) to get,
\begin{equation*} 
  \forall\,m\leq2,\qquad  \| D^m \uv(t, \cdot) \|_{0, 2} \le C \left ( { 1 \over \alpha}
    \|\uv_0\|_{0,2} + { 1 \over \alpha^3} E_{\alpha, 0} \sqrt { t \over \nu } \right ), 
\end{equation*} 
concluding the proof.
\end{proof} 
We note that the Duhamel principle applies, by Lemma~8 in~\cite{Ler1934}. Therefore, we
have for all $0 < \tau \le t < T$,
\begin{equation*} 
  \uv (t, \x) = (Q \star \uv(\tau, \cdot) ) (t, \x) + \int_{\tau}^{t} \inte
  \g T (t - t', \x- \yv ): (\overline{\uv\otimes \uv}) (t', \yv) \,d\yv dt'.
\end{equation*}
Therefore, from this principle, by induction and following the same proofs of
Lemma~\ref{lem:pression_regularity}, Lemma~\ref{lem:velocity_regularity}, and
Lemma~\ref{lem:regularity_vitesse}, and Lemma~\ref{eq:differentiation}, we can easily show
the following result.
 \begin{lemma}
   \label{lem:regularité-HM} 
   For all $0 < \tau_1 < \tau_2 < T$ and for $m > 2$, we have that $ (\uv, p) \in
   C([\tau_1, \tau_2]; H^m(\R^3)^3 \times H^m(\R^3))$ and there exists $C =
   C(m,\tau_1,T,\| \uv_0\|_{0, 2}, \nu, \alpha)$ such that 
   \begin{equation}
     \label{eq:estimate_Hm} 
     \forall \, t \in [\tau_1, \tau_2], \qquad \| \uv
     (t, \cdot) \|_{m,2}\le C(m, \tau_1, T, \| \uv_0\|_{0, 2}, \nu, \alpha). 
   \end{equation}
 \end{lemma}
 We stress that the constant in~(\ref{eq:estimate_Hm}) depends on $T$ as
 $\mathcal{O}(\sqrt T)$, and remains finite, whatever the value of $T<\infty$ is
 considered. We also notice that we cannot let $\tau_1$ to go to zero when $m>2$, which is
 due to the Helmholtz filter that only regularizes the initial data up to $H^2(\R^{3})$
 and not better, since $\uv_0 \in L^2(\R^3)^3$, which is the appropriate regularity
 assumption about the initial value in the framework of the Navier-Stokes equations.  This
 is one of the main differences with respect to the filtering made by convolution with
 smooth functions.
\begin{lemma}
  \label{lem:continuation} 
  Let us assume $T>0$. Any regular solution on $[0, T[$ can be extended up to $t=T$.
\end{lemma}
\begin{proof} 
  Let $0 < t_1 < t_2 <T$. We write
\begin{equation*} 
  \uv (t_2, \x) = (Q \star \uv(t_1, \cdot) )
  (t_2, \x) + \int_{t_1}^{t_2} \inte \g T (t_2 - t', \x- \yv ): (\overline{\uv\otimes \uv})
  (t', \yv)\,d \yv dt'. 
\end{equation*} 
Therefore, by the same calculation as above, using the again the energy balance, we get
\begin{equation} 
  \label{eq:cauchy_criterium} 
  \| \uv(t_1, \cdot) - \uv (t_2, \cdot) \|_{2,
    2}\le \| (Q \star \uv (t_1, \cdot))(t_2, \cdot) - \uv(t_1, \cdot) \|_{2,2} + {C \over
    \alpha^3} E_{\alpha, 0} \sqrt {t_2-t_1 \over \nu}.  
\end{equation} 
Standard results about the heat equation yield the estimate
\begin{equation*} 
  \| (Q \star \uv (t_1, \cdot))(t_2, \cdot) - \uv(t_1, \cdot) \|_{2,2} \le
  {\nu (t_2-t_1 ) \over 2} \| \uv(t_1, \cdot) \|_{3, 2}.  
\end{equation*} 
Therefore, by Lemma~\ref{lem:regularité-HM}, which provides a uniform bound in $t_1$ about
$\| \uv(t_1, \cdot) \|_{3, 2}$ for $t_1 \ge \tau >0$ (the time $\tau <T$ being fixed),
combined with~(\ref{eq:cauchy_criterium}), we see that $\uv (t, \cdot)$ satisfies a
uniform Cauchy criterion in the Banach space $H^2(\R^3)^3$ over $[\tau, T[$, in particular
when both $t_1 , t_2 \to T$. Therefore, $\uv(t, \cdot)$ admits a limit when $t \to T$,
concluding the proof.
\end{proof}
\subsection{Iterative procedure}
\label{sec:Iterations} 
Let $\tau >0$ be a given time, which will be fixed later. We equip the Banach space
$C([0,\tau]; H^2(\R^3)^3)$ with its natural uniform norm 
\begin{equation*} 
  \|\vv (t, \cdot) \|_{\tau; 2,2} := \sup_{t \in [0, \tau]} \| \vv(t,\cdot) \|_{2,2}.
\end{equation*} 
We consider in $C([0, \tau]; H^2(\R^3)^3)$ the sequence $(\uv^{(n)})_{n \in \N}$ defined
by
\begin{equation} 
  \label{eq:ite_def} 
  \left \{ 
    \begin{array}{l} 
      \zoom \uv^{(0)}(t, \x) =
      \overline {\uv_0}(\x), \phantom{\int_0^1} 
      \\ 
      \uv^{(n)} (t, \x) = (Q \, \star \, \overline{
        \uv_0}) (t, \x) + 
      \\ 
      \zoom \hskip 3,5cm \int_0^t \inte \g {\bf T} (t-t', \x - \yv) :
      \overline{\uv^{(n-1)} \otimes \uv^{(n-1)} } (t', \yv)\, d \yv dt'.  
    \end{array}
  \right. 
\end{equation} 
By the same reasoning of the previous section, it is easily checked by induction that each
$\uv^{(n)}$ lies indeed in $C([0, \tau]; H^2(\R^3)^3)$ (but in fact and much better
space). In the following, $C$ denotes a constant that only depends on Sobolev constants
and the Oseen tensor. Let $t \mapsto E^{(n)}_\alpha(t)$ be the function defined by
\begin{equation*}
  E^{(n)}_\alpha(t) := \| \uv^{(n)} (t, \cdot) \|_{0,2}^2+\alpha^2 \| \g \uv^{(n)}
  (t,\cdot) \|_{0,2}^2.   
\end{equation*} 
For technical conveniences, we will assume in the following that $\alpha \le
1$\footnote{We can take any other upper bound for $\alpha$, which only yields technical
  complications}.

The goal is to prove that the sequence $(\uv^{(n)})_{n \in \N}$ satisfies a contraction
  property on a time interval $[0, \tau_{Lip}]$. We will therefore conclude to
the convergence of the sequence $(\uv^{(n)})_{n \in \N}$ by the Picard theorem.  To do so,
we need to estimate $E_\alpha^{(n)}(t)$, at least over a small time interval $[0, \tau[$
to begin. This is the purpose of the next lemma.
\begin{lemma}
  Let us define
  \begin{equation}
    \label{eq:tau_m}
    \tau_{max}(\sigma):= { \nu \alpha^6 \over 4C^2 \sigma}\quad\text{for }\sigma>0. 
  \end{equation} 
  Then
  \begin{equation} 
    \label{eq:bound_E_alpha} 
    \forall \, n \in \N \text{ and } \forall \, t \in
    [0, \tau_{max}(E_{\alpha, 0})], \qquad E^{(n)}_\alpha(t) \le 8 E_{\alpha,
      0}. 
  \end{equation}
\end{lemma}
\begin{proof}
  We argue by induction. We have for all $t \in [0, \infty[$,
  $E^{(0)}_\alpha(t)=E_{\alpha, 0}\le 8 E_{\alpha, 0}$. Let $n \ge 0$ be given.  On one
  hand we have
\begin{equation} 
  \label{eq:ite1} 
  \| \uv^{(n)} (t, \cdot)\|_{0, 2}\le \|\ID \|_{0, 2}+
  \int_0^t { \|(\overline { \uv^{(n-1)} \otimes \uv^{(n-1)} }) (t', \cdot) \|_{0,2} \over
    \sqrt {\nu(t-t')} }\,dt', 
\end{equation} 
and, by~(\ref{eq:norm2_1}), we also have for $t' \in [0, t]$
\begin{equation} 
  \label{eq:ite2}
  \begin{array}{ll} 
    \|(\overline { \uv^{(n-1)} \otimes
      \uv^{(n-1)} }) (t', \cdot) \|_{0,2} & \le C \alpha^{-{3 \over 2}} \| (\uv^{(n-1)} \otimes
    \uv^{(n-1)}) (t', \cdot) \|_{0,1} 
    \\
    & \le \zoom C \alpha^{-{3 \over 2}} \|
    \uv^{(n-1)}(t', \cdot) \|^2_{0, 2}.
  \end{array}
\end{equation} 
On the other hand, we have
\begin{equation} 
  \label{eq:ite3}
  \| \g \uv^{(n)} (t, \cdot)\|_{0, 2}\le \|\g \ID \|_{0,
    2}+ \int_0^t { \|\g (\overline { \uv^{(n-1)} \otimes \uv^{(n-1)} }) (t', \cdot) \|_{0,2}
    \over \sqrt {\nu(t-t')} }\,dt', 
\end{equation} 
and consequently, by using Sobolev inequality,
\begin{equation} 
  \label{eq:ite4} 
  \begin{array}{ll} 
    \|\g (\overline { \uv^{(n-1)} \otimes
      \uv^{(n-1)} }) (t', \cdot) \|_{0,2} & \le C \alpha^{-{1}} \| (\uv^{(n-1)} \otimes
    \uv^{(n-1)}) (t', \cdot) \|_{0,2} 
    \\ 
    & \le \zoom C \alpha^{-1} \| \uv^{(n-1)} \|^2_{0,
      4} \phantom{\int_0^1}
    \\
    & \le \zoom C \alpha^{-1} \left [ \| \uv^{(n-1)} \|^2_{0, 2}
      + \| \g \uv^{(n-1)} \|^2_{0, 2} \right ].
\end{array}
\end{equation} 
We then combine~(\ref{eq:ite1}),~(\ref{eq:ite2}),~(\ref{eq:ite3}), and~(\ref{eq:ite4})
with elementary algebraic inequalities, and we get
\begin{equation} 
  \label{eq:ite5}
  \sqrt {E_\alpha^{(n)} (t) } \le \sqrt {2 E_{\alpha, 0}} +
  {C \over \alpha^3}\int_0^t {E_\alpha^{(n-1)} (t') \over \sqrt { \nu (t-t')} }\,dt'= F(
  E_\alpha^{(n-1)})(t).
\end{equation} 
Then, we deduce from~(\ref{eq:ite5}) by induction, that the inequality ``$E_\alpha^{(n-1)}
(t) \le 8 E_{\alpha, 0}$'' yields ``$E_\alpha^{(n)} (t) \le 8 E_{\alpha, 0}$'' for all $t
\in [0, \tau]$, and for any time $\tau>0$ such that
\begin{equation*} 
  F( 8 E_{\alpha, 0} )(\tau) \le 2 \sqrt 2 \sqrt{ E_{\alpha, 0} }.
\end{equation*} 
After easy calculations (where we use $\alpha \le 1$ to balance the terms of the form
$\alpha^{-s}$) we can explicitly prove the upper bound
\begin{equation*} 
  \tau \le {\nu \alpha^6 \over 4 C^2 E_{\alpha, 0}}= \tau_{max}(E_{\alpha,0}), 
\end{equation*} 
hence the result.
\end{proof} 
It is very important to notice that the function $\sigma \mapsto \tau_{max}(\sigma
)$ is
non-increasing. From this local time estimate of $E_\alpha^{(n)} (t)$, we can now prove
the following result.
\begin{lemma}
  \label{lem:contraction}
  There exists $\tau_{Lip} = \tau_{Lip}(E_{\alpha, 0}) \in \, ]0, \tau_{max}(E_{\alpha, 0})]$ such that
  \begin{equation*}
    \forall \, n \ge 1, \qquad \| \uv^{(n+1)} - \uv^{(n)}\|_{\tau_{Lip}; 2,2}
    \le {1 \over 2} \,\| \uv^{(n)} - \uv^{(n-1)}\|_{\tau_{Lip}; 2,2}. 
  \end{equation*} 
  Moreover, the function $\sigma \mapsto \tau_{Lip}(\sigma)$ is a non-increasing function.
\end{lemma}
\begin{proof} Let $t \le \tau_{max}(E_{\alpha, 0})$, in order to
  apply~(\ref{eq:bound_E_alpha}). We deduce from~(\ref{eq:calcul_integral})
  and~(\ref{eq:ite_def}) that
\begin{equation} 
  \label{eq:ite_est_1} 
  \| (\uv^{(n+1)} - \uv^{(n)}) (t, \cdot) \|_{2,2}\le
  \int_0^t { \|(\overline { \uv^{(n)}\otimes \uv^{(n)} - \uv^{(n-1)}\otimes \uv^{(n-1)})
      (t', \cdot) }\|_{2, 2} \over \sqrt {\nu(t-t')} }\,dt'.
\end{equation} 
Repeating the same reasoning as above yields, for $t' \in [0, t]$,
\begin{equation} 
  \label{eq:ite_est_2} 
  \begin{array}{l} 
    \|(\overline { \uv^{(n)}\otimes
      \uv^{(n)} - \uv^{(n-1)}\otimes \uv^{(n-1)}) (t', \cdot) }\|_{2, 2} 
    \\
    \zoom \hskip 1cm\phantom {\int_0^1} \leq {C \over \alpha}\left [ \| \uv^{(n)}(t', \cdot) \|_{1, 2} + \|
      \uv^{(n-1)}(t', \cdot) \|_{1, 2} \right ] \| (\uv^{(n)} - \uv^{(n-1)}) (t', \cdot)
    \|_{1,2}.
 \end{array}
\end{equation} 
By~(\ref{eq:bound_E_alpha}), we have $\| \uv^{(p)}(t', \cdot) \|_{1, 2} \le \sqrt
2\alpha^{-1} \sqrt{E_\alpha(t')} \le 4 \alpha^{-1} \sqrt{E_{\alpha, 0}} $, in particular
for $p=n, n-1$. Therefore, combining~(\ref{eq:ite_est_1}) and~(\ref{eq:ite_est_2}) we
obtain that, for all $\tau \in \, ]0, \tau_{max}(E_{\alpha,
  0})]$, and for all $t \in [0, \tau]$,
\begin{equation*}
\| (\uv^{(n+1)} - \uv^{(n)}) (t, \cdot) \|_{2,2}\le {4C \over \alpha^2}\sqrt{\tau
  E_{\alpha, 0}\over \nu} \| \uv^{(n+1)} - \uv^{(n)} \|_{\tau; 2,2},
\end{equation*}
leading to
\begin{equation*}
 \| \uv^{(n+1)} - \uv^{(n)}  \|_{\tau; 2,2}\le {4C \over \alpha^2}\sqrt{\tau E_{\alpha,
     0}\over \nu}  \| \uv^{(n+1)} - \uv^{(n)} \|_{\tau; 2,2} .
\end{equation*}
The conclusion follows by taking
\begin{equation} 
  \label{eq:tau_k}
  \tau_{Lip} (E_{\alpha, 0}) =\inf \left\{ {\nu \alpha^4 \over
      16 C^2 E_{\alpha, 0}}, \tau_{max}(E_{\alpha, 0}) \right\}. 
\end{equation}
By~(\ref{eq:tau_m}) and~(\ref{eq:tau_k}), we see that the function $\sigma \mapsto
\tau_{Lip}(\sigma)$ is a non-increasing function.
\end{proof} 
We deduce from Lemma~\ref{lem:contraction} and Picard's Theorem,  that the sequence
$(\uv^{(n)})_{n \in \N}$ is convergent in the Banach space $C([0, \tau_{Lip}]; H^2(\R^3)^3)$
to some $\uv$. To conclude the proof of Theorem~\ref{thm:existence_solution_reg} it
remains to prove that this field is indeed a solution to the NSEB-$\alpha$ model. This is
the aim of the next section.
\subsection{Existence of a regular solution}
In this subsection we prove the final results leading to the existence of a regular
solution. This subsection is divided into three steps:
\begin{enumerate}[i)]
\item \label{it:Fin_step_1} we prove that $\uv$ satisfies the
  relation~(\ref{eq:pression_reg});
\item \label{it:Fin_step_2} we show that there exists $p \in C([0, \tau_{Lip}]; H^2(\R^3)^3)$,
  such that $(\uv, p)$ is a regular solution of the NSEB-$\alpha$ model on the time
  interval $[0, \tau_{Lip}[$;
\item \label{it:Fin_step_3} we show that $(\uv, p)$ can be extended for all $t \in [0,
\infty[$.
\end{enumerate}

\ref{it:Fin_step_1}) We first pass to the limit in the recursive
relation~(\ref{eq:ite_def}), that can be written in an abstract way as $\uv^{(n+1)} =
F(\uv^{(n)})$, and our aim is to prove $\uv = F(\uv)$. This will also prove the continuity
of the function $\vv \mapsto F(\vv)$ in the Banach space $C([0, \tau_{Lip}]; H^2(\R^3)^3)$.

\smallskip

Let $(t, \x) \in [0, \tau_{Lip}] \times \R^3$ be fixed, and for $(t', \yv) \in [0, t[ \times
\R^3$, we set:
\begin{equation*} 
  \left \{ 
    \begin{array}{lcl} \zoom \psi_n (t, \x , t', \yv) &= &\g T
      (t-t', \x- \yv) : \overline{\uv^{(n)} \otimes \uv^{(n)} } (t', \yv),
      \\ 
      \psi (t, \x , t',
      \yv) &=& \zoom \g T (t-t', \x- \yv) : \overline{\uv \otimes \uv } (t', \yv),
      \phantom{\inte \psi_n (t, \x ; t', \yv)} 
      \\ 
      \varphi_n(t,t', \x)& = & \zoom \inte \psi_n(t, \x ; t', \yv) \, d \yv, 
      \\ 
      \\
      \varphi (t,t', \x)& = & \zoom \inte \psi (t, \x ; t', \yv)\,d \yv.
    \end{array} 
  \right. 
\end{equation*} 
We must check that
\begin{equation} 
  \label{eq:leb_appl} 
  \lim_{n \to \infty}\int_0^t \varphi_n(t,t', \x)\,dt'= \int_0^t \varphi(t,t', \x) \,dt'. 
\end{equation} 
To do so we will apply Lebesgue dominated convergence Theorem twice.  From the previous
results we deduce that
\begin{equation*}
  \lim_{n \to \infty}\psi_n (t, \x ; t', \yv) = \psi (t, \x ; t', \yv) =
  \g T (t-t', \x- \yv) : \overline{\uv \otimes \uv } (t', \yv), 
\end{equation*} 
uniformly in $\yv \in \R^3$, for any given $t' \in [0, t[$. Moreover, as $(\uv^{(n)})_{n
  \in \N}$ is a sequence converging in $C([0, \tau_{Lip}]; H^2(\R^3)^3)$, it is bounded in the
same space. Consequently, by the results of Section~\ref{sec:helmholz_filter}, it is
easily checked that
\begin{equation*} 
  | \psi_n (t, \x ; t', \yv) |\le C \sup_{n \in \N}\|\uv^{(n)}
  \|_{\tau_{Lip}; 2,2}^2 |\g T (t-t', \x- \yv)|\in L^1_{\yv} (\R^3), 
\end{equation*}
by~(\ref{eq:calcul_integral}), since $t'<t$. Therefore, we have by Lebesgue's theorem,
\begin{equation*} 
  \lim_{t \to \infty} \varphi_n(t,t', \x) = \varphi (t,t',\x). 
\end{equation*} 
Similarly,
\begin{equation*}
  | \varphi_n(t,t', \x) | \le {C \over \sqrt { \nu (t-t') } } \in L^1([0,t]),
 \end{equation*} 
 hence~(\ref{eq:leb_appl}) holds by Lebesgue's theorem once again.  In conclusion, $\uv$
 satisfies the integral relation~(\ref{eq:pression_reg}) as claimed.  \medskip

\ref{it:Fin_step_2}) Let us now consider the unsteady (linear) Stokes problem:
\begin{equation} 
  \label{eq:Stokes_Bardina} 
  \left \{ 
    \begin{array}{ll}
      \p_t \vv - \nu
    \Delta \vv + \g p = - \div ( \overline {\uv \otimes \uv}) &\quad \hbox{in} \quad [0,
    \tau_{Lip}] \times \R^3, 
    \\
    \div \, \vv = 0 &\quad \hbox{in} \quad [0, \tau_{Lip}] \times \R^3, 
    \\
    \vv_{t=0} = \overline {\uv_0}&\quad \hbox{in} \quad \R^3.
  \end{array} 
\right.
\end{equation} 
This Stokes problem has a source term in $C([0, \tau_{Lip}]; H^3(\R^3)^3)$ and an initial data
in $H^2(\R^3)^3$. By standard results (see for instance Caffarelli, Kohn, and
Nirenberg~\cite{CKN1982}, Solonnikov~\cite{VS64}, and Temam~\cite{RT01}) we already know
the existence of a unique variational solution $(\vv, p)$ to the
problem~(\ref{eq:Stokes_Bardina}), obtained by the Galerkin method, and with (at least)
the regularity
\begin{equation*}
  \begin{array}{l}  
    \vv \in C([0, \tau_{Lip}]; H^2(\R^3)^3 ) \cap L^2([0, \tau_{Lip}]; H^4(\R^3)^3
    ), \quad \p_t \vv \in L^2([0, \tau_{Lip}]; H^2(\R^3)^3 ), 
    \\
    \zoom p \in L^2 ([0, \tau_{Lip}]; H^1(\R^3)^3 ). \phantom{\int_0^t} 
  \end{array}
\end{equation*}
 From this,
it is easily checked that $(\vv, p)$ is a strong solution to ~(\ref{eq:Stokes_Bardina}),
therefore a regular solution in the sense of Definition~\ref{def:regular}, which satisfies
by Lemma 8 in~\cite{Ler1934} the integral formulation
\begin{equation*}
\vv (t, \x) =  (Q \, \star \,   \overline{ \uv_0}) (t, \x)  + \int_0^t \inte \g  {\bf T}
(t-t', \x - \yv) :   \overline{\uv  \otimes \uv} (t', \yv)\, d \yv dt'.
\end{equation*}
Hence, $\uv = \vv$, and $(\uv, p)$ is indeed a regular solution to the NSEB-$\alpha$ model
over $[0, \tau_{Lip}]$.  \medskip

\ref{it:Fin_step_3}) It remains to check that the solution $\uv$ can be extended up to
$[0, \infty[$. We already know by Lemma~\ref{lem:continuation} that $\uv$ can be extended
to $t = \tau_{Lip}$. We also know by Corollary~\ref{cor:energy_balance} that this solution
satisfies the energy balance. Therefore, the function $t \mapsto E_{\alpha} (t)$ is non
increasing over $[0, \tau_{Lip}]$, and we have in particular
\begin{equation} 
  \label{eq:energy_consequence} 
  E_{\alpha, 1}= E_\alpha (\tau_{Lip}(E_{\alpha, 0})) \le E_{\alpha, 0}.
 \end{equation} 
 The construction carried out in Subsection~\ref{sec:Iterations} and
 step~\ref{it:Fin_step_1}) can be reproduced starting from the time $t= \tau_{Lip} (E_{\alpha,
   0})$ instead of $ t= 0$, and with initial data $\uv (\tau_{Lip} (E_{\alpha, 0}), \cdot)$
 instead of $\overline {\uv_0}$. We then get a regular solution to the NSEB-$\alpha$ model
 over the time interval $[\tau_{Lip} (E_{\alpha, 0}), \tau_{Lip} (E_{\alpha, 0})+ \tau_{Lip}
 (E_{\alpha, 1})]$. As the solution is left continuous at $t = \tau_{Lip} (E_{\alpha, 0})$ in
 $H^2(\R^3)^3$, and also right continuous at the same time, it is continuous, and
 therefore we constructed
\begin{equation*}
  \uv \in C([0,  \tau_{Lip} (E_{\alpha, 0})+ \tau_{Lip} (E_{\alpha, 1})]; H^3(\R^3)^3).
\end{equation*}
We first observe that this $\uv$ is a weak solution to the NSEB-$\alpha$, but then we can
easily check from the equation that the gluing at $t = \tau_{Lip} (E_{\alpha, 0})$ is of class
$C^1$ in time, so that we get a regular solution over $[0, \tau_{Lip} (E_{\alpha, 0})+
\tau_{Lip}(E_{\alpha, 1})]$. The main point that allows to iterate this process, is that the
functions $\sigma \mapsto \tau_{Lip} (\sigma)$ and $t \mapsto E_\alpha(t)$ are both
non-increasing. In particular we get by~(\ref{eq:energy_consequence}),
\begin{equation*} 
  \tau_{Lip} (E_{\alpha, 0}) \le \tau_{Lip} (E_{\alpha, 1}). 
\end{equation*} 
This suggests to build the sequence $(T_n)_{n \in \N}$, by setting
\begin{equation*}
  T_0 =  \tau_{Lip} (E_{\alpha, 0}),
\end{equation*}
and assuming that we have constructed a regular solution of the NSEB-$\alpha$ model over
$[0, T_n]$. Then we extend the solution as above, starting from $t = T_n$ and $\uv(T_n,
\cdot)$ over the time interval $[T_n, T_n + \tau_{Lip}(E_{\alpha, n})]$, where $E_{\alpha, n
}= E_\alpha(T_n)$. Therefore,
\begin{equation*}
T_{n+1}= T_n + \tau_{Lip}(E_{\alpha, n}).
\end{equation*}
 Since 
 \begin{equation*}
 \tau_{Lip}(E_{\alpha, n}) \ge \tau_{Lip}(E_{\alpha, n-1}) \ge \cdots \ge \tau_{Lip}(E_{\alpha, 0}),
 \end{equation*}
we have
\begin{equation*}
T_n \ge n \tau_{Lip}(E_{\alpha, 0}),
\end{equation*}
hence
\begin{equation*}
 \lim_{n \to \infty}T_n = \infty,
\end{equation*}
and the solution is indeed constructed for all $t \in [0, \infty[$.
\section*{Acknowledgments} 
The research that led to the present paper was partially supported by a grant of the group
GNAMPA of INdAM 
\bibliographystyle{plain} 
\bibliography{Biblio}
\end{document}